\documentclass[graybox]{svmult}

\usepackage{type1cm}        
\usepackage{makeidx}         
\usepackage{graphicx}        
\usepackage{multicol}        
\usepackage[bottom]{footmisc}

\usepackage{newtxtext}       %
\usepackage{newtxmath}       

\makeindex             

\begin{document}

\title*{Conditions of exact null controllability and the problem of complete stabilizability for time-delay systems}
\titlerunning{Exact controllability and complete stabilizability for time-delay systems} 
\author{Pavel Barkhayev, Rabah Rabah and Grigory Sklyar}
\institute{Pavel Barkhayev \at Norwegian University of Science and Technology, Høgskoleringen 1, 7491, Trondheim, Norway 
	and B.Verkin Institute for Low Temperature Physics and Engineering,
	Academy of Sciences of Ukraine, 47 Nauki Ave., 61103 Kharkiv, Ukraine, \email{pavloba@ntnu.no}
\and Rabah Rabah \at IMT Atlantique, Mines-Nantes, 4 rue Alfred Kastler, BP 20722 Nantes, France, \email{rabah3.14159@gmail.com}
\and Grigory Sklyar \at Institute of Mathematics, University of Szczecin, Wielkopolska 15, 70-451, Szczecin, Poland, \email{sklar@univ.szczecin.pl}}
%
%
\maketitle

\abstract*{For a class of linear time-delay control systems satisfying the property of completability of the generalized eigenvectors we prove that the problems of complete stabilizability and exact null controllability are equivalent.}

\abstract{For a class of linear time-delay control systems satisfying the property of completability of the generalized eigenvectors we prove that the problems of complete stabilizability and exact null controllability are equivalent.}

\section{Introduction}\label{sec:intro}\label{sec:1}

The problems of controllability and stabilizability are one of the central and most investigated in the mathematical control theory; these problems are well-studied in many cases.
The relation between the problems as well as these notions themselves depend essentially on specific settings.
For example, in finite-dimensional settings most of the main controllability notions (exact, null, approximate) are equivalent and imply stabilizability. Moreover, complete stabilizability (stabilizability with an arbitrary decay rate) is equivalent to controllability.
In infinite-dimensional settings the situation is much more sophisticated.
First of all, different notions of controllability are not equivalent in general. Secondly, 
exact null controllability implies complete  stabilizability but the inverse does not hold in general.  
In some special cases, if ${\rm e}^{\mathcal{A} t}$ is a group (\cite{Zabczyk_1992}) 
or the operators ${\rm e}^{\mathcal{A} t}$ are surjective for all $t\ge 0$ (\cite{Rabah_Karrakchou_1997,Zeng_Yi_Xie_2013}), complete stabilizability implies exact controllability. 
Finally, exact null controllability implies complete stabilizability with bounded feedback (see \cite{Rabah_Sklyar_Barkhayev_2017}  for the proof and counterexamples).

The problem of exponential stabilizability for linear time-delay systems was considered by many authors, see e.g. \cite{Hale_Verduyn_1993,Pandolfi_1976,Oconnor_Tarn_1983,Richard_2003,Michiels_Niculescu_2007} and references therein.
For the problem of asymptotic non-exponential  stabilizability, which appears only for neutral type systems, 
we refer to \cite{Rabah_Sklyar_Rezounenko_2008,Rabah_Sklyar_Barkhayev_2012}.
An analysis of relations between exact controllability and exponential stabilizability  may be found e.g. in \cite{Salamon_1984b,Ito_Tarn_1985,Oconnor_Tarn_1983,Dusser_Rabah_2001}.

A rather general class of linear control time-delay systems with distributed delays is described by equation
\begin{equation}\label{eq:intro_01} 
\dot{z}(t) - A_{-1} \dot{z}(t-1) = L z_t(\cdot) + b u(t), \quad t\ge 0,
\end{equation}
here $z(t) \in \mathbb{R}^n$ is state, and  we use the notation 
$z_t(\theta)=z(t+\theta)$, $\theta\in [-1,0]$, $t\ge 0$; $u(t) \in \mathbb{R}$ is control, and the matrices $A_{-1}$, $A_{2}$, $A_{3}$ and $b$ are of appropriate dimensions; the elements of $A_{2}$ and  $A_{3}$ take values in $L^2(-1,0)$; 
the operator $L$, acting from the Sobolev space $W^{1,2}((-1,0); \mathbb{R}^n)$ to $\mathbb{R}^n$  is bounded, i.e.
$$
Lz_t(\cdot)= \int^0_{-1}\left [A_2(\theta)\dot z(t+\theta) + A_3(\theta) z(t+\theta)\right ]{\mathrm d}\theta.
$$

For the system (\ref{eq:intro_01}) of neutral type the problem of exact controllability is deeply investigated 
in \cite{Rabah_Sklyar_2007,Rabah_Sklyar_Barkhayev_2016}.
It was shown that controllability set coincides with the domain of the operator corresponding to the system (\ref{eq:intro_01}) and exact null controllability is equivalent to the rank conditions
\begin{eqnarray}
\mathrm{rank}\begin{pmatrix}
\Delta_{\mathcal A}(\lambda); & b
\end{pmatrix} = n \; \mbox{ for all } \; \lambda \in \mathbb{C}, \label{eq:spec_contr} \\
\mathrm{rank}\begin{pmatrix}
-\mu I + A_{-1}; & b
\end{pmatrix} = n  \; \mbox{ for all } \; \mu \in \mathbb{C}, \label{eq:A_1_contr}
\end{eqnarray}
where $\Delta_{\mathcal A}(\lambda)$ is the characteristic matrix of the system~(\ref{eq:intro_01}):
\begin{equation}\label{eq:Delta}
	\Delta_{\mathcal A}(\lambda)= -\lambda I
	+ \lambda e^{-\lambda}A_{-1} + 
	\int^0_{-1}\left [\lambda e^{\lambda s} A_2(s)  +e^{\lambda
		s} A_3(s)\right]{\mathrm d}s.
\end{equation}
We note that the condition~(\ref{eq:A_1_contr}) assures the existence of a feedback change of the form $u(t)=v(t)+p_{-1}\dot{z}(t-1)$ such that the matrix $A_{-1}+bp_{-1}$ is non-singular. This means that the operator $\mathcal{A}$ is a generator of a $C_0$-group and thus exact null controllability is equivalent to exact controllability.

Later in \cite[Theorem 6]{Rabah_Sklyar_Barkhayev_2017} it was shown that complete stabilizability is equivalent to the condition~(\ref{eq:spec_contr}) together with
\begin{equation}\label{eq:A_1_contr_sing}
\mathrm{rank}\begin{pmatrix}
-\mu I + A_{-1}; & b
\end{pmatrix} = n  \; \mbox{ for all } \; \mu \in \mathbb{C}\backslash\{0\}.
\end{equation}
This means, in particular, that exact controllability implies complete stabilizability for the systems~(\ref{eq:intro_01}) and if the matrix $A_{-1}$ is non-singular or if $\mathrm{rank}\begin{pmatrix} A_{-1}; & b \end{pmatrix} = n$
then exact controllability is equivalent to complete stabilizability.
However, in the case $\mathrm{rank}\begin{pmatrix} A_{-1}; & b \end{pmatrix} < n$
the situation is unclear. 
In \cite{Rabah_Sklyar_Barkhayev_2017} we posed a conjecture that complete stabilizability is equivalent to exact null controllability in the general case.

In the special case of retarded systems ($A_{-1}=0$) complete stabilizability is equivalent to~(\ref{eq:spec_contr}) which is called spectral controllability condition.
It is well-known that exact null controllability implies spectral controllability (see e.g. \cite{Salamon_1984}), however the inverse have not been proved for general systems.
One of the first results on this issue was obtained in \cite{Jacobs_Langenhop_1976} by Jacobs and Langenhop, where the conjecture is proved in case $n=2$ and $L f = A_1 f(-h) + A_0 f(0)$, $A_i\in\mathbb{R}^{2\times 2}$.
Later, in \cite{Marchenko_1979} Marchenko claimed the conjecture for control systems with finitely many delays, 
however, in \cite{Salamon_1984} it is  noticed that his arguments seem to be incomplete.
In 1984 Colonius \cite{Colonius_1984} has showed the conjecture in the case  $L f = A_1 f(-h) + A_0 f(0)$ and arbitrary $n$. His proof is based on the fact that spectrum controllability is equivalent to solvability of finite spectrum assignment problem.
Later Olbrot and Pandolfi \cite{Olbrot_Pandolfi_1988} have given an explicit algebraic algorithm of computing a control function 
which steers any given initial function to zero in finite time for quite wide class of retarded control systems.

In the present paper we show that the property of completeness (completability) of the set of generalized eigenvectors is crucial for equivalence between exact controllability and complete stabilizability.
This property allows to represent the steering conditions of controllability as a vector moment problem and investigate its solvability in an appropriate class of functions. 
If the system (\ref{eq:intro_01}) is of neutral type then the family of exponential corresponding to the moment problem form the Riesz basis of it closure what gives a powerful tool of investigation (\cite{Avdonin_Ivanov_1995}).
This method was used in  \cite{Rabah_Sklyar_2007,Rabah_Sklyar_Barkhayev_2016}
where an exhaustive analysis of controllability was given.
In case of systems of retarded or mixed type the system of exponentials does not possess the Riesz basis property. 
However, as it was shown in \cite{Khartovskii_Pavlovskaya_2013} for the case of systems with point-wise delays some similar conditions of controllability may be obtained. 
We use the moment problem approach and show that some trigonometric moment problems are solvable on intervals of appropriate length.
To show solvability we construct explicitly systems biorthogonal to exponentials (\cite{Avdonin_Ivanov_1995,Levin_1996}).
For a class of completable systems this method allows to prove the equivalence of exact null controllability and complete stabilizability and, thus, to show the conjecture posed in \cite{Rabah_Sklyar_Barkhayev_2017}.

The paper is organized as follows. In Sect.~\ref{sec:pre} we introduce some notations and definitions, and rewrite equivalently the system (\ref{eq:intro_01}) in the infinite-dimensional model space.
In Sect.~\ref{sec:mp} we introduce the moment problem corresponding to the controllability conditions. We give the conditions of completeness and completability of the generalized eigenvectors.
In Sect.~\ref{sec:res} we prove the main result on controllability which implies equivalence of complete stabilizability and exact null controllability for system satisfying the completability condition. Besides, we give examples.

\section{Preliminaries}\label{sec:pre}

Let us consider an initial condition of the form
\begin{equation}\label{eq:intro_02}
\left\{
\begin{array}{l}
z(0)=y,\\
z(t)=z_0(t), \quad t\in [-1,0).
\end{array}
\right.
\end{equation}
In the case $A_{-1}=0$, for any $y\in\mathbb{R}^n$, $z_0(t)\in L^2((-1,0); \mathbb{R}^n)$
and any control $u(t)\in L^2_{\rm loc}[0,+\infty)$ there exists a unique solution $z(t)$, $t\ge 0$ of the initial-value problem~(\ref{eq:intro_01}),(\ref{eq:intro_02}), which is continuous (see  \cite{Delfour_1980}).
In the case of neutral type systems ($A_{-1}\not=0$) the existence of the strong solution is guaranteed for smooth initial states only: $z_0(t)\in W^{1,2}((-1,0); \mathbb{R}^n)$  (see \cite{Burns_Herdman_Stech_1983}, where a neutral operator of more general form was considered).

These facts naturally leads to consideration the problem (\ref{eq:intro_01}),(\ref{eq:intro_02}) in the product space 
$$
M_2(-1,0;\mathbb{R}^n) \stackrel{\mathrm{def}}{=} \mathbb{R}^n\times L^2(-1,0;\mathbb{R}^n),
$$
further noted shortly as $M_2$.
Thus, the initial-value problem~(\ref{eq:intro_01})-(\ref{eq:intro_02}) may be rewritten as
\begin{equation}\label{eq:intro_03}
\left\{
\begin{array}{l}
\dot{x}(t) = \mathcal{A} x(t) + \mathcal{B} u(t), \; t\ge 0,\\
x(0) = x_0,
\end{array}
\right.
\end{equation}
here 
$$
{\mathcal A} x(t)=
\begin{pmatrix}
Lz_t(\cdot) \cr
\frac{{\mathrm d} z_t(\theta)}{{\mathrm d}\theta}
\end{pmatrix}, \quad
x(t)= \begin{pmatrix}
y(t)\cr  z_t(\cdot)
\end{pmatrix},
$$
with 
$$
\mathcal{D}(\mathcal{A}) =  
\left\{ 
\left(
\begin{array}{c}
y\\
z_0(\tau)
\end{array}
\right): \; z_0(\tau)\in W^{1,2}((-1,0); \mathbb{R}^n), y = z_0(0) - A_{-1}z_0(-1)
\right\},
$$
and $\mathcal{B} u(t) = \left(
\begin{array}{c}
b u(t)\\
0
\end{array}
\right) \in M_2$,
$x_0 = \left(
\begin{array}{c}
y\\
z_0(\tau)
\end{array}
\right)$.

\begin{definition}\label{def:stabilizab}
	Time-delay system (\ref{eq:intro_01}) (or the infinite-dimensional system (\ref{eq:intro_03})) is said to be exponentially stabilizable if there is a linear feedback operator ${\mathcal F}$: 
$$
	u(t)= {\mathcal F} x(t) = F_{-1}\dot z(t-1) + \int^0_{-1}\left [F_2(\theta)\dot z(t+\theta) + F_3(\theta) z(t+\theta)\right ]{\mathrm d}\theta,
$$
	such that the semigroup 
	${\mathrm e}^{({\mathcal A} +{\mathcal B} {\mathcal F})t}$ is exponentially stable, i.e. there is a $\omega > 0$ such that
	\begin{equation}    \label{eq:exp-stab}
	\left \Vert {\mathrm e}^{({\mathcal A} +{\mathcal B} {\mathcal F})t}\right\Vert \le M_{\omega}{\mathrm e}^{-\omega t}, \quad M_{\omega} \ge 1.
	\end{equation}
	The system is said to be  completely stabilizable (or stabilizable with an arbitrary decay rate) if for all $\omega >0$ there is 
	a linear bounded feedback ${\mathcal F}_\omega$ such that (\ref{eq:exp-stab}) holds.
\end{definition}

\begin{definition}
	An initial state $x_0 =(y, z_0(\tau)) \in M_2$ is said to be null controllable by means of the system~(\ref{eq:intro_01}) at time $T$ 
	if there exists a control $u(t)\in L^2[0,T]$ such that $z(t)=z(t, y, z_0, u(t))\equiv 0$ for $t\in [T-1, T]$.
\end{definition}

Since the evolution of the Cauchy problem~(\ref{eq:intro_03}) is described by 
$$
x(t)={\rm e}^{\mathcal{A}t}x_0 +\int_0^t {\rm e}^{\mathcal{A}(t-\tau)}\mathcal{B} u(\tau){\rm d}\tau,
$$
then the null controllable states satisfy the relation
\begin{equation}\label{eq:pre_01}
{\rm e}^{\mathcal{A}T}x_0 = - \int_0^T {\rm e}^{\mathcal{A}\tau}\mathcal{B}u(T-\tau)\: {\rm d}\tau,
\end{equation}
what naturally leads to the notion of attainable set at time $T$:
\begin{equation}\label{eq:pre_02}
\mathcal{R}_T = 
\left\{
\int_0^T {\rm e}^{\mathcal{A}\tau} \mathcal{B}  u(\tau)\: {\rm d}\tau: \; u(t)\in L^2[0,T]
\right\} \subset M_2.
\end{equation}

\begin{definition}
	System~(\ref{eq:intro_01}) is said to be exactly null controllable from $F\subset M_2$ at time $T$ if for any $x_0\in$:
	${\rm e}^{\mathcal{A}T} F \subset \mathcal{R}_T$.
\end{definition}

\begin{definition}
	System~(\ref{eq:intro_01}) is said to be exactly null controllable at time $T$ if it is controllable from $M_2$, i.e.
	${\rm Im} ({\rm e}^{\mathcal{A}T}) \subset \mathcal{R}_T$.
\end{definition}

Not every system of the form~(\ref{eq:intro_01}) may be exactly null controllable.
If e.g. $\det A_{-1}\not=0$ then $F=\mathcal{D}(\mathcal{A})$ is the maximal possible set of null controllability
(see \cite{Rabah_Sklyar_2007,Rabah_Sklyar_Barkhayev_2016} for more details).
However, retarded systems ($A_{-1}=0$) can be exactly null controllable from $M_2$.


\section{The moment problem and completability property}\label{sec:mp}

The eigenvalues $\lambda\in\sigma(\mathcal{A})$ are zeros of the characteristic function $\det\Delta_\mathcal{A}(\lambda)$, where $\Delta_\mathcal{A}(\lambda)$ is given by (\ref{eq:Delta}).
Let $\lambda\in\sigma(\mathcal{A})$ then the corresponding eigenvector $\varphi_\lambda$ of the operator $\mathcal{A}$ is 
of the form
$$
\varphi_\lambda=
\left(
\begin{array}{c}
(I-{\rm e}^{\lambda}A_{-1})x_\lambda\\ 
{\rm e}^{\lambda \tau}x_\lambda
\end{array}
\right),\quad x_\lambda\in \mathrm{Ker}\Delta(\lambda),
$$
and the eigenvector $\psi_\lambda$ of the operator $\mathcal{A}^*$ corresponding to $\overline{\lambda}$ is of the form
\begin{equation}\label{eq:psi}
\psi_\lambda=
\left(
\begin{array}{c}
y_\lambda\\ 
\left[\overline{\lambda} {\rm e}^{-\overline{\lambda} \tau}I -A_2^*(\tau) 
+ \int_0^\tau {\rm e}^{\overline{\lambda} (s-\tau)} \left(A_3^*(s)
+ \overline{\lambda} A_2^*(s)\right)\:{\rm d}s \right]y_\lambda
\end{array}
\right),
\end{equation}
where $y_\lambda\in \mathrm{Ker}\Delta^*(\lambda)$.

In (\ref{eq:pre_01}) we make the change of function $v(\tau)=-u(T-\tau)$ and multiply the relation by $\psi_\lambda$:
$$
\left\langle {\rm e}^{\mathcal{A}T}x_0, \psi_\lambda \right\rangle_{M_2}
= \int_0^T \left\langle {\rm e}^{\mathcal{A}\tau}\mathcal{B} v(\tau), \psi_\lambda \right\rangle_{M_2} {\rm d}\tau,
$$
what gives 
$$
\left\langle x_0, {\rm e}^{\overline{\lambda}T} \psi_\lambda \right\rangle_{M_2}
= \int_0^T \left\langle \mathcal{B} v(\tau), {\rm e}^{\overline{\lambda}\tau}\psi_\lambda \right\rangle_{M_2} {\rm d}\tau,
$$
or 
$$
{\rm e}^{\lambda T} \left\langle x_0, \psi_\lambda \right\rangle_{M_2}
= \left\langle b, y_\lambda \right\rangle_{\mathbb{R}^n} \int_0^T {\rm e}^{\lambda\tau} v(\tau) {\rm d}\tau.
$$

If the spectral controllability condition~(\ref{eq:spec_contr}) holds then $\left\langle b, y_\lambda \right\rangle_{\mathbb{R}^n}\not=0$,
and we obtain the moment problem
\begin{equation}\label{eq_m_07}
s_\lambda  =  \int_0^T {\rm e}^{\lambda\tau} v(\tau) {\rm d}\tau, 
\quad \lambda\in \sigma(\mathcal{A}),
\end{equation}
where
\begin{equation}\label{eq_m_07aaa}
s_\lambda \equiv {\rm e}^{\lambda T} \left\langle x_0, \psi_\lambda \right\rangle_{M_2} 
(\left\langle b, y_\lambda \right\rangle_{\mathbb{R}^n})^{-1}.
\end{equation}

\begin{remark}
If $u(t)\in L^2(0,T)$ steers $x_0\in M_2$ to zero at time $T$ then $v(t)$ solves the moment problem~(\ref{eq_m_07}).
The inverse assertions holds only if the system $\{\psi_\lambda\}_{\lambda\in \sigma(\mathcal{A})}$ is complete in $M_2$.
\end{remark}
 Indeed, the sequence $\{s_\lambda\}$ defines uniquely the initial point $x_0$ by~(\ref{eq_m_07aaa}),
since if $s_\lambda = {\rm e}^{\lambda T} \left\langle x_0, \psi_\lambda \right\rangle_{M_2} 
(\left\langle b, y_\lambda \right\rangle_{\mathbb{R}^n})^{-1} = {\rm e}^{\lambda T} \left\langle x_1, \psi_\lambda \right\rangle_{M_2} 
(\left\langle b, y_\lambda \right\rangle_{\mathbb{R}^n})^{-1}$ then $\left\langle x_0 - x_1, \psi_\lambda \right\rangle_{M_2}=0$,
$\lambda\in \sigma(\mathcal{A})$,  what gives $x_1=x_0$.
Further, if $x_0\not\in R_T$ it would mean that 
$h\equiv {\rm e}^{\mathcal{A}T}x_0 - \int_0^T {\rm e}^{\mathcal{A}\tau}\mathcal{B} u(T-\tau) {\rm d}\tau \not = 0$.
However, $\left\langle h, \psi_\lambda \right\rangle_{M_2}=0$ due to~(\ref{eq_m_07}), thus $h=0$.

Further we give the conditions of completeness.
Let $m_\lambda$ be the length of the maximal chain of the generalized eigenvectors corresponding to $\lambda\in\sigma(\mathcal{A})$.
The generalized eigenspace corresponding to $\lambda$ is
$$
V_\lambda(\mathcal{A}) = \mathrm{Ker} (\mathcal{A} - \lambda I)^{m_\lambda},
$$
and the linear span 
$$
V(\mathcal{A}) = {\rm Lin}\{V_\lambda(\mathcal{A}):\; \lambda\in\sigma(\mathcal{A})\}
$$
is called the generalized eigenspace of $\mathcal{A}$.

The entire function $f(\lambda)=\det\Delta(\lambda)$ is of exponential type and satisfies the condition:
$$
\frac{\log|f(x_0+iy)|}{1+y^2}\in L^1(\mathbb{R}),
$$
for some fixed $x_0$ satisfying $x_0> {\rm Re}\lambda$, for any $\lambda\in\sigma(\mathcal{A})$.
This means that $\det\Delta(\lambda)$ belongs to the class $C$ (see e.g. \cite[Lecture 16]{Levin_1996} for more details).
It is known that indicator function of $f\in C$ if of the form
$$
h_f(\theta)=\lim_{r\rightarrow\infty}\frac{\log|f(r{\rm e}^{i\theta})|}{r}
=
\left\{
\begin{array}{ll}
\alpha_-(f) |\cos\theta|, & \frac{\pi}{2}\le \theta \le \frac{3\pi}{2},\\
\\
\alpha_+(f) \cos\theta, & \frac{3\pi}{2}\le \theta \le \frac{5\pi}{2},
\end{array}
\right.
$$
where $\alpha_-(f)$, $\alpha_+(f)$ are some constants and the limit exists almost for all $\theta$.

In \cite{Verduyn_Yakubovich_1997} a criterion of completeness of generalized eigenvectors is established (in more general settings).
\begin{theorem}{\cite[Theorem~4.2]{Verduyn_Yakubovich_1997}}\label{thm_ly}
	The generalized eigenvectors of $\mathcal{A}$ are complete in $M_2$, i.e. $\overline{V(\mathcal{A})} = M_2$
	if and only if 
	$$\alpha_-(\det\Delta) = n.$$
\end{theorem}

Let us consider the special case of operator $L$:
\begin{equation}\label{eq_m_08a}
L f = A_1f(-1)+ \int^0_{-1} A_2(\theta)f'(\theta)\:{\rm d}\theta 
+ \int^0_{-1} A_3(\theta)f(\theta) \:{\rm d}\theta,
\end{equation}
assuming that 
\begin{equation}\label{eq_m_08b}
{\rm supp}\: A_i(\theta)\subset [\alpha, 0], \quad \alpha > -1, i=2,3.
\end{equation}
For this case Theorem~\ref{thm_ly} may be reformulates as

\begin{corollary}\label{col_02a}
	The generalized eigenvectors of the operator $\mathcal{A}$ of the system~(\ref{eq:intro_01}) 
	with the operator $L$ defined by~(\ref{eq_m_08a}) 
	are complete if and only if 
	the matrix pencil $A_1 + \lambda A_{-1}$ is non-singular, i.e. there exists such $\lambda_0\in\mathbb{C}^n$ 
	that 
	$$\det(A_1 + \lambda_0 A_{-1})\not = 0.$$
\end{corollary}
\begin{proof}
	We rewrite
	$$
	\det\Delta(\lambda) = f_1(\lambda) + f_2(\lambda), \quad f_1(\lambda) = \det(A_1 + \lambda A_{-1}) {\rm e}^{-nh}.
	$$
	Due to the assumption~(\ref{eq_m_08b}): $\alpha_-(f_2) < n$ and since $\alpha_-({\rm e}^{-nh}) = n$ we can conclude
	that $\alpha_-(\det\Delta) = n$ if and only if  $\det(A_1 + \lambda A_{-1})\not\equiv 0$.
\end{proof}

It is worth to mention that this result may be obtained using the technique similar to \cite[Theorem~2]{Bartosiewicz_1980}.

\begin{definition}\label{def:completability}
	Control system~(\ref{eq:intro_01}) is said to be completable if there exists a feedback 
	$$
	u(t)= \mathcal{P} x(t) = p_{-1} \dot{z}(t-1) + P z_t,\quad p_{-1}\in\mathbb{R}^{n}, P\in \mathfrak{B}(W^{1,2}((-1,0), \mathbb{R}^n), \mathbb{R}),
	$$
	such that the operator $\mathcal{A}+\mathcal{B}\mathcal{P}$ of the closed-loop system possesses complete
	set of generalized eigenvectors.
\end{definition}

From Definition~\ref{def:completability} and Corollary~\ref{col_02a} we obtain:
\begin{corollary}\label{col:completability}
	The system~(\ref{eq:intro_01}) 
	with the operator $L$ defined by~(\ref{eq_m_08a})  is completable if and only if 
	there exists a feedback of the form
	$$
	u(t)= p_{-1} \dot{z}(t-1) + p_1 z(t-1),\quad p_{-1}, p_1\in\mathbb{R}^{n},
	$$
	such that the matrix pencil 
    $$
	(A_1 + b p_1) + \lambda (A_{-1} + b p_{-1})
	$$
	is non-singular.
\end{corollary}

\section{The main result}\label{sec:res}

Consider the special case of the system~(\ref{eq:intro_01}) when $A_{-1}=0$ and the operator $L$ is defined by~(\ref{eq_m_08a}), i.e.:
\begin{equation}\label{eq_bs_01_01zz}
\dot{z}(t)=A_1 z(t-1) + \int_{-1}^0 \left[A_2(\theta) \dot{z}(t+\theta) + A_3(\theta) z(t+\theta)\right] {\rm d}t + b u(t),
\end{equation}
with
\begin{equation}\label{eq_m_08bzz}
{\rm supp}\: A_i(\theta)\subset [\alpha, 0], \quad \alpha > -1, \quad i=2,3.
\end{equation}
We assume that the system (\ref{eq_bs_01_01zz}) is completable, what means, due to Corollary \ref{col:completability}, that the pair $(A_1, b)$ is controllable:
\begin{equation}\label{eq_bs_01_01szz}
{\rm rank}(- \lambda I + A_1; b) = n, \quad \lambda\in \mathbb{C}.
\end{equation}

\begin{theorem}
The system (\ref{eq_bs_01_01zz}) satisfying the rank condition (\ref{eq_bs_01_01szz}) is exactly null-controllable from any initial state $x_0\in M_2$ if and only if it is spectrally controllable.
\end{theorem}
\begin{proof}
We divide the proof of the proposition into several steps.
	
{\em Step 1.} We fix $n$ distinct nonzero numbers $\{a_i\}$ and apply a change of feedback $u(t) = v(t) + \hat{A}_1 z(t-1)$ and a change of coordinates such that the system takes the form (\ref{eq_bs_01_01zz}) with
$$
	A_1 = {\rm diag}(a_i), \quad b=(1,\ldots, 1)^*.	
$$
Besides we assume, without loss of generality, that the eigenvalues of the corresponding operator $\mathcal{A}$ are simple. Indeed, only a finite number of eigenvalues may be multiple, thus there exists a feedback change of the form 
$$
	u(t)=v(t)+ \int_{-1}^0 \left[\hat{A}_2(\theta) \dot{z}(t+\theta) + \hat{A}_3(\theta) z(t+\theta)\right] {\rm d}t
$$
that makes the spectrum of $\mathcal{A}$ simple.

{\em Step 2.} The characteristic function of the system $\dot{z}(t)=A_1 z(t-1)$
has the form $\det(-\lambda I + {\rm e}^{-\lambda} A_1) = \prod_{j=1}^n (-\lambda + a_j{\rm e}^{-\lambda})$. 
Each of the entire functions $g_j(\lambda) = -\lambda + a_j {\rm e}^{-\lambda}$
possesses infinitely many zeros which we denote as $\{\tilde{\lambda}_k^j\}_{k\in \mathbb{Z}}$ 
assuming that $\lambda^j_0$ is the only real zero of $g_j(\lambda)$ and ${\rm Im} \lambda^j_k < {\rm Im} \lambda^j_s $ as $k<s$. 
We also consider the circles $L_k^j (r_0)$ of fixed radius $r_0= \frac{1}{3}\min \{
|\tilde{\lambda}_{k_1}^{j_1} -\tilde{\lambda}_{k_2}^{j_2} |, (k_1,j_1)\neq (k_2,j_2)\}$ centered
at $\tilde{\lambda}_k^j$.
 
Similarly to \cite[Theorem 2]{Rabah_Sklyar_Rezounenko_2005} it can be shown that there exists $N\in\mathbb{N}$ such that for any $j$, satisfying $|j|\ge N$, the characteristic function $\det \Delta_{\mathcal A}(\lambda)$ of (\ref{eq_bs_01_01zz})
possesses a zero in the circle $L^j_k(r_0)$.

We denote the zeros of the characteristic equation $\det \Delta_{\mathcal A}(\lambda)$ as $\{\lambda_k^j\}$.
The eigenvectors $\{{\psi}_k^j\}$ of the operator $\mathcal{A}^*$ are given by (\ref{eq:psi}).
Due to the spectral controllability condition we have $\left\langle b, y_k^j \right\rangle_{\mathbb{R}^n}\not=0$ and thus we can normalize ${\psi}_k^j$ such that $\left\langle b, y_k^j \right\rangle_{\mathbb{R}^n}=1$.
Then the moment problem (\ref{eq_m_07}) for equation (\ref{eq_bs_01_01zz}) takes the form
\begin{equation}\label{eq_m_08qq}
s_k^j  =  \int_0^T {\rm e}^{{\lambda}_k^j\tau} u(\tau) {\rm d}\tau,
\quad  j=\overline{1,n}, k\in\mathbb{Z},
\end{equation}
where 
$$
s_k^j\equiv {\rm e}^{\lambda_k T}  \left\langle x_0, {\psi}_k^j \right\rangle_{M_2}.
$$

The system $\{{\psi}_k^j\}$ is complete  in $M_2$ due to Corollary~\ref{col_02a}, thus solvability of (\ref{eq_m_08qq}) is equivalent to exact null controllability of initial state $x_0\in M_2$.
Further we show that the moment problem (\ref{eq_m_08qq}) is solvable in class $L^2[0,T]$ for $T>T_0$.

{\em Step 3.} Consider an entire function $g(\lambda) = - \lambda + a {\rm e}^{-\lambda}$, $\lambda\in\mathbb{C}$. We denote its zeros as $\{\lambda_k\}_{k\in \mathbb{Z}}$ 
assuming that $\lambda_0$ is the only real zero of $g(\lambda)$ and ${\rm Im} \lambda_k < {\rm Im} \lambda_s $ as $k<s$, and consider the family of exponentials
\begin{equation}\label{eq_mp_01}
\{{\rm e}^{\lambda_k t}, \quad t\in [0,T],\; \lambda_k: \lambda_k {\rm e}^{\lambda_k}=1\}.
\end{equation}
There exists $T_0>0$ such that the family (\ref{eq_mp_01}) is minimal in $L^2[0, T]$ for any $T>T_0$ (see e.g. \cite{Young_1980}) . 
Thus, there exists a biorthogonal family to (\ref{eq_mp_01}) in $L^2[0, T]$, and we construct it explicitly.

First we define $\widetilde{f}_s(t)={\rm e}^{-\overline{\lambda}_s t} -  {\rm e}^{-\overline{\lambda_{s+1}} t}$,  $s\in\mathbb{Z}$,and 
\begin{equation}\label{eq_mp_08}
\widetilde{f}_s^1(t)
=\left\{ 
\begin{array}{ll}
\widetilde{f}_s(t), & t\in[0,1],\\
0, & t\in(1,T],
\end{array}
\right.
\qquad
\widetilde{f}_s^2(t)
=\left\{ 
\begin{array}{ll}
0, & t\in[0,T-1),\\
\widetilde{f}_s(t-T+1), & t\in[T-1,T].
\end{array}
\right.
\end{equation}
Further we construct in the space $L^2[0, T]$ ($T>1$) the biorthogonal to $\{{\rm e}^{\lambda_k t}\}_{k\in\mathbb{Z}}$ family $\{{f}_s(t)\}_{s\in\mathbb{Z}}$:
\begin{equation}\label{eq_mp_07}
f_s(t)=\beta_s^1 \widetilde{f}_s^1(t) + \beta_s^2 \widetilde{f}_s^2(t), \quad \beta_s^i\in \mathbb{C},
\end{equation}
where $\beta_s^i$ are obtained from the condition $\langle {\rm e}^{\lambda_k t}, f_s(t)\rangle_{L^2[0,T]}  = \delta_{ks}$:
\begin{equation}\label{eq_mp_09}
\beta_s^1 = -\frac{\widetilde{\lambda}_s}{(\widetilde{\lambda}_{s+1} - \widetilde{\lambda}_s)\alpha_s},\qquad
\beta_s^2 = \frac{\widetilde{\lambda}_s \widetilde{\lambda}_{s+1}}{(\widetilde{\lambda}_{s+1} - \widetilde{\lambda}_s)\alpha_s}, 
\end{equation}
where $\widetilde{\lambda}_s = (\lambda_s)^{T-1}$, $\alpha_s = 1+\frac{1}{\lambda_s}$.

We estimate the growth of the constructed family $\{f_s(t)\}$:
$$
\|f_s(t)\|_{L^2[0,T]}\le |\beta_s^1| \|\widetilde{f}_s^1(t)\|_{L^2[0,T]} + |\beta_s^2| \|\widetilde{f}_s^2(t)\|_{L^2[0,T]}
= \left(|\beta_s^1| + |\beta_s^2|\right) \|\widetilde{f}_s(t)\|_{L^2[0,1]}.
$$
Since $|\alpha_s|\sim 1$, $|\lambda_s - \lambda_{s+1}|\sim 1$ and $|\lambda_s|\sim s$ one has
$
|\beta_s^1|\sim s^{T-1}$,and $|\beta_s^2|\sim s^{2(T-1)}$.
The norm of $\widetilde{f}_s(t)$ may be computed explicitly:
$$
\|\widetilde{f}_s(t)\|^2_{L^2[0,1]} =
= \frac{- \left[ (|\lambda_s|^2 - 1) {\rm Re} \lambda_{s+1} + (|\lambda_{s+1}|^2 - 1) {\rm Re} \lambda_s \right] 
	|\lambda_s - \lambda_{s+1}|^2}
{2{\rm Re} \lambda_s {\rm Re} \lambda_{s+1} |\overline{\lambda}_s + \lambda_{s+1}|^2}.
$$
Since $|{\rm Re} \lambda_s|\sim \ln s$ and $|\overline{\lambda}_s + \lambda_{s+1}|\sim 1$ one has
\begin{equation}\label{eq_mp_11}
\|\widetilde{f}_s(t)\|_{L^2[0,1]}\sim  s (\ln s)^{-\frac12}.
\end{equation}
Thus, there exist $C>0$ and $s_0\in\mathbb{N}$ such that for any $s: |s|>s_0$ one has
\begin{equation}\label{eq_mp_12}
\|f_s(t)\|_{L^2[0,T]}\le C s^{1+2(T-1)} (\ln s)^{-\frac12}.
\end{equation}

{\em Step 4.} We consider the family of exponentials corresponding to eigenvalues $\{\tilde{\lambda}_k^j\}$, ${k\in \mathbb{Z}}$, $j=\overline{1,n}$ constructed above:
\begin{equation}\label{eq_mp_13}
\{{\rm e}^{\lambda_k^j t}:\;  t\in [0,T],\; \lambda_k^j {\rm e}^{\lambda_k^j}=a_j, j=\overline{1,n}, k\in\mathbb{Z}\},
\end{equation}
and the corresponding moment problem
\begin{equation}\label{eq_mp_14}
s_k^j = \int_0^T {\rm e}^{\lambda_k^j\tau} u(\tau) {\rm d}\tau,  \quad j=\overline{1,r}, k\in\mathbb{Z}.
\end{equation}
The family (\ref{eq_mp_13}) if minimal in $L^2[0, T]$ for $T>n$ (see e.g. \cite{Young_1980}).
Thus, there exists a biorthogonal family and below we construct and estimate it, what allows to show solvability of the moment problem (\ref{eq_mp_14}) in $L^2[0,T]$.

Let us introduce the Hilbert space
\begin{equation}\label{eq_mp_16}
{E}_1 = \left\{x: x=(x_1,\ldots,x_r), x_j \in {\rm Lin}\{ {\rm e}^{\lambda_k^j t}, t\in [0,T] \} \right\} 
\subset \prod\limits_{j=1}^r L^2[0,T]
\end{equation}
with the norm of space $E\equiv \prod\limits_{j=1}^r L^2[0,T]$: $\|x\|_{\widetilde{E}_1 }=\sum_{j=1}^r \|x_j\|_{L^2[0,T]}$, and two families of functions
\begin{equation}\label{eq_mp_17}
\begin{array}{l}
\varphi_k^j(t) = \left(\underbrace{0, \ldots, 0}_{j-1}, {\rm e}^{\lambda_k^j t},0, \ldots, 0 \right) \in E_1,\\
\psi_k^j(t) = \left(\underbrace{0, \ldots, 0}_{j-1}, f_k^j,0, \ldots, 0 \right) \in \prod\limits_{j=1}^r L^2[0,T],
\end{array}
\end{equation}
where for each fixed $j$, the family $\{f_k^j\}$, constructed by (\ref{eq_mp_07}) is biorthogonal to $\{{\rm e}^{\lambda_k^j t}\}$ in $L^2[0,T]$.
By construction we have
$\left\langle \varphi_{k_1}^{j_1},\; \psi_{k_2}^{j_2} \right\rangle_{E} = \delta_{(k_1, j_1)}^{(k_2, j_2)}$.

Let $\widetilde{D}$ be a linear operator acting from ${E}_1$ to $L^2[0,T]$ as follows: 
\begin{equation}\label{eq_mp_18}
\widetilde{D} x = \sum_{j=1}^r x_j.
\end{equation}
Obviously, the operator $\widetilde{D}$ is bounded:
\begin{equation}\label{eq_mp_19}
\|\widetilde{D} x\|_{L^2[0,T]} \le\sum_{j=1}^r  \|x_j\|_{L^2[0,T]} = \|x\|_{E_1},
\end{equation}
thus, it may be extended onto $\overline{E}_1$, the closure of ${E}_1$, without increasing of the norm, i.e.
there exists the operator $D: \overline{E}_1 \rightarrow L^2[0,T]$ such that
\begin{equation}\label{eq_mp_20}
Dx = \widetilde{D} x, x \in E_1, \quad \mbox{and } \quad \|D\| = \|\widetilde{D}\|.
\end{equation}
Due to construction the range of $D$ is the following subspace of $L^2[0, T]$:
\begin{equation}\label{eq_mp_21}
E_2 = \overline{{\rm Lin}\{ {\rm e}^{\lambda_k^j t}:\;  t\in [0, T]\}} .
\end{equation}
Moreover, since ${\rm e}^{\lambda_k^j t}$ is minimal in $L^2[0, T]$, 
then $D$ is bijective as an operator from $\overline{E}_1$ to $E_2$ 
and due to the Banach theorem on the inverse operator $D^{-1}: E_2\rightarrow \overline{E}_1$ is bounded.

We denote by $\widetilde{\psi}_k^j$ the orthogonal projection of $\psi_k^j$ onto subspace $\overline{E}_1$, 
i.e. $\psi_k^j = \widetilde{\psi}_k^j + d_k^j$, where $d_k^j\in (\overline{E}_1)^\perp$.
The family $\{\widetilde{\psi}_k^j\}$ is biorthogonal to $\{\varphi_{k_1}^{j_1}\}$ in $\overline{E}_1$:
\begin{equation}\label{eq_mp_22}
\left\langle \varphi_{k_1}^{j_1},\; \widetilde{\psi}_{k_2}^{j_2} \right\rangle_{\overline{E}_1} = 
\left\langle \varphi_{k_1}^{j_1},\; \psi_{k_2}^{j_2} \right\rangle_{E} = \delta_{(k_1, j_1)}^{(k_2, j_2)},
\end{equation}
moreover, it is well-known that
$\left\| \widetilde{\psi}_{k}^{j} \right\|_{\overline{E}_1} \le \left\| \psi_{k}^{j} \right\|_{E}$.

The family $\{g_k^j(t)\}=\{(D^{-1})^*\widetilde{\psi}_k^j\}$ is biorthogonal to $\{{\rm e}^{\lambda_{k}^{j} t}\}$ in $E_2\subset L^2[0,T]$:
\begin{equation}\label{eq_mp_23}
\left\langle {\rm e}^{\lambda_{k_1}^{j_1} t},\; (D^{-1})^*\widetilde{\psi}_{k_2}^{j_2} \right\rangle_{E_2} = 
\left\langle D^{-1} {\rm e}^{\lambda_{k_1}^{j_1} t},\; \widetilde{\psi}_{k_2}^{j_2} \right\rangle_{\overline{E}_1} = 
\left\langle \varphi_{k_1}^{j_1},\; \widetilde{\psi}_{k_2}^{j_2} \right\rangle_{\overline{E}_1} = \delta_{(k_1, j_1)}^{(k_2, j_2)}.
\end{equation}

Since $D^{-1}$ is bounded, then the adjoin operator $(D^{-1})^*$ is bounded as well and $\|(D^{-1})^*\| = \|D^{-1}\|$,
and thus, we can estimate
\begin{equation}\label{eq_mp_24}
\left\| g_k^j(t) \right\|_{L^2[0,T]} \le \left\|(D^{-1})^*\right\| \left\| \widetilde{\psi}_{k}^{j} \right\|_{\overline{E}_1} \le
C \left\| \psi_{k}^{j} \right\|_{E} = C \left\| f_{k}^{j} \right\|_{L^2[0,T]}.
\end{equation}

{\em Step 5.}  The sequences $\{{\lambda}_k^j\}$ and $\{\tilde{\lambda}_k^j\}$ are such that $|{\lambda}_k^j - \tilde{\lambda}_k^j|\rightarrow 0$ as $k\rightarrow\infty$ for any $j\in\overline{1,n}$.
Let $T>0$ is such that in $L^2(0,T)$ the system $\{{\rm e}^{\tilde{\lambda}_k^j t}\}$ is minimal
and $\{g_k^j(t)\}$ is the biorthogonal system constructed above.
Then there exists a system $\{h_k^j(t)\}\subset L^2(0,T)$ biorthogonal to $\{{\rm e}^{{\lambda}_k^j t}\}$ such that
$$
\|g_k^j(t) - h_k^j(t)\|_{L^2(0,T)} \rightarrow 0, \quad k\rightarrow \infty.
$$

Due to the estimate on  $\{s_k\}$  the series 
\begin{equation}\label{eq_m_08cc}
u(t)= \sum_k s_k^j \tilde{h}_k^j(t) 
\end{equation}
converges and $u(t)$ solves the problem  (\ref{eq_m_08qq}).
\end{proof}

\begin{corollary}
The system (\ref{eq_bs_01_01zz}) satisfying the rank condition (\ref{eq_bs_01_01szz}) is completely stabilizable if and only if it is exactly null controllable.
\end{corollary}

\begin{remark}
	The completability condition (\ref{eq_bs_01_01szz}) is not necessary for exact null controllability. 
\end{remark}	
	Indeed, let us consider the system
	$$
	\dot{z}(t) = 
	\left( \begin{array}{ll}
	0 & 0\\ 
	0 & 1
	\end{array} \right) z(t-1)
	+
	\left( \begin{array}{ll}
	0 & 1\\ 
	0 & 0
	\end{array} \right) z(t)
	+
	\left( \begin{array}{l}
	0\\ 
	1
	\end{array} \right) u(t).
	$$
	It is not completable, however for any initial state it is possible to construct a control steering it to zero.

\begin{remark}
	Completeness (completability) does not imply spectral controllability. 
\end{remark}	
The system
	$$
	\dot{z}(t) = 
	\left( \begin{array}{ll}
	1 & 0\\ 
	0 & 1
	\end{array} \right) z(t-1)
	+
	\left( \begin{array}{l}
	0\\ 
	1
	\end{array} \right) u(t)
	$$
is complete, however it is not controllable.

\begin{acknowledgement}
Pavel Barkhayev was supported by the Norwegian Research Council project ''COMAN'' No. 275113.
\end{acknowledgement}

%
%

\begin{thebibliography}{99.}%

\bibitem{Avdonin_Ivanov_1995}
Avdonin, S.A., Ivanov, S.A.: Families of exponentials.  The method of moments in controllability problems for distributed parameter systems. Cambridge University Press, Cambridge (1995)

\bibitem{Bartosiewicz_1980}
Bartosiewicz, Z.: Density of images of semigroup operators for linear neutral
functional differential equations. J. Differential Equations, \textbf{38(2)}, 161--175 (1980)


\bibitem{Burns_Herdman_Stech_1983}
Burns, J.A., Herdman, T.L., Stech, H.W.: Linear functional-differential equations as semigroups on product spaces. SIAM J. Math. Anal., \textbf{14(1)}, 98--116 (1983)

\bibitem{Colonius_1984}
Colonius, F.: On approximate and exact null controllability of delay systems.
Systems Control Lett., \textbf{5(3)}, 209--211 (1984)

\bibitem{Delfour_1980}
Delfour, M. C.: The largest class of hereditary systems defining a $C\sb{0}$ semigroup on the product space.
Canadian J. Math., \textbf{32(4)}, 969--978 (1980)

\bibitem{Dusser_Rabah_2001}
Dusser, X., Rabah, R.: On exponential stabilizability of linear neutral systems. Math.
	Probl. Eng.,  \textbf{7(1)}, 67--86 (2001)

\bibitem{Hale_Verduyn_1993}
Hale, J.K., Verduyn~Lunel, S.M.: Introduction to functional-differential equations. Volume~99 of Applied Mathematical Sciences. Springer-Verlag, New York, (1993)

\bibitem{Ito_Tarn_1985}
Ito, K., Tarn, T.~J.: A linear quadratic optimal control for neutral systems. Nonlinear Anal., \textbf{9(7)}, 699--727 (1985)

\bibitem{Jacobs_Langenhop_1976}
Jacobs, M., Langenhop, C.E.: Criteria for function space controllability of linear neutral
systems. SIAM J. Control Optimization, \textbf{14(6)}, 1009--1048 (1976)

\bibitem{Khartovskii_Pavlovskaya_2013}
Khartovskii, V.E., Pavlovskaya, A.T.: Complete controllability and controllability for linear autonomous
systems of neutral type. Automation and Remote Control, \textbf{(5)}, 769--784 (2013)

\bibitem{Levin_1996}
Levin, B.Ya.: Lectures on entire functions. Volume 150 of Translations of Mathematical Monographs. American Mathematical Society, Providence, RI (1996)

\bibitem{Marchenko_1979}
Marchenko, V.M.: On the controllability of zero function of time lag systems.
Problems Control Inform. Theory/Problemy Upravlen. Teor. Inform., \textbf{8(5-6)}, 421--432 (1979)

\bibitem{Michiels_Niculescu_2007}
Michiels, W., Niculescu, S.-I.: Stability and stabilization of time-delay systems. An
	eigenvalue-based approach. Vol.~12 of Advances in Design and Control,
Society for Industrial and Applied Mathematics, Philadelphia, PA (2007)

\bibitem{Oconnor_Tarn_1983}
O'Connor, D.~A., Tarn, T.~J.: On stabilization by state feedback for neutral differential
equations. IEEE Trans. Automat. Control, \textbf{28(5)}, 615--618 (1983)

\bibitem{Olbrot_Pandolfi_1988}
Olbrot, A.W., Pandolfi, L.: Null controllability of a class of functional-differential systems.
Internat. J. Control, \textbf{47(1)}, 193--208 (1988)

\bibitem{Pandolfi_1976}
Pandolfi, L.: Stabilization of neutral functional differential equations. J. Optimiz. Theory Appl., \textbf{20(2)}, 191--204 (1976)

\bibitem{Rabah_Karrakchou_1997}
Rabah, R., Karrakchou, J.: On exact controllability and complete stabilizability for linear
	systems in Hilbert spaces, Applied Mathematics Letters, \textbf{10(1)}, 35--40 (1997)

\bibitem{Rabah_Sklyar_Rezounenko_2005}
Rabah, R., Sklyar, G.M., Rezounenko, A.V.: Stability analysis of neutral type systems in Hilbert space.
J. Differential Equations, \textbf{214(2)}, 391--428 (2005)

\bibitem{Rabah_Sklyar_2007}
Rabah, R., Sklyar, G.M.: The analysis of exact controllability of neutral-type systems by the
moment problem approach. SIAM J. Control Optim., \textbf{46(6)}, 2148--2181 (2007)

\bibitem{Rabah_Sklyar_Rezounenko_2008}
Rabah, R., Sklyar, G.M., Rezounenko, A.V.: On strong regular stabilizability for linear neutral type systems. J. Differential Equations,  \textbf{245(3)}, 569--593 (2008)

\bibitem{Rabah_Sklyar_Barkhayev_2012}
Rabah, R., Sklyar, G.~M., Barkhayev, P.~Yu.: Stability and stabilizability of mixed retarded-neutral type systems. ESAIM Control Optim. Calc. Var., \textbf{18(3)}, 656--692 (2012)

\bibitem{Rabah_Sklyar_Barkhayev_2016}
Rabah, R., Sklyar, G.M., Barkhaev, P.Yu.: On the problem of the exact controllability of neutral-type systems with delay. Ukrain. Mat. Zh., \textbf{68(6)}, 800--815 (2016)

\bibitem{Rabah_Sklyar_Barkhayev_2017}
Rabah, R., Sklyar, G.M., Barkhaev, P.Yu.: Exact null controllability, complete stabilizability and continuous final observability of neutral type systems. Int. J. Appl. Math. Comput. Sci., \textbf{27(3)}, 489--499 (2017)

\bibitem{Richard_2003}
Richard, J.-P.: Time-delay systems: an overview of some recent advances and open
problems. Automatica J. IFAC, \textbf{39(10)}, 1667--1694 (2003)

\bibitem{Salamon_1984}
Salamon, D.: On controllability and observability of time delay systems.
IEEE Trans. Automat. Control, \textbf{29(5)}, 432--439 (1984)

\bibitem{Salamon_1984b}
Salamon, D.: Control and observation of neutral systems. Volume~91 of Research Notes in Mathematics, Pitman, Boston, MA (1984)

\bibitem{Verduyn_Yakubovich_1997}
Verduyn~Lunel, S.M., Yakubovich, D.V.: A functional model approach to linear neutral functional-differential equations. Integral Equations Operator Theory, \textbf{27(3)}, 347--378 (1997)

\bibitem{Young_1980}
Young, R.M.: An introduction to nonharmonic Fourier series. Volume~93 of Pure and Applied Mathematics. Academic Press Inc., New York (1980)

\bibitem{Zabczyk_1992}
Zabczyk, J.: Mathematical control theory: an introduction. Systems \&
Control: Foundations \& Applications, Birkh\"auser Boston, Inc., Boston, MA (1992)

\bibitem{Zeng_Yi_Xie_2013}
Zeng, Y., Xie, Z., Guo, F.: On exact controllability and complete stabilizability for linear
systems. Appl. Math. Lett.,  \textbf{26(7)}, 766--768 (2013)

%
\end{thebibliography}
%

\end{document}